\documentclass[12pt, twoside, leqno]{article}
\usepackage{hyperref}

\usepackage{amsmath,amsthm}
\usepackage{amssymb}

\usepackage{enumitem}

\usepackage{graphicx}

\usepackage[T1]{fontenc}

\newtheorem{theorem}{Theorem}[section]
\newtheorem{corollary}[theorem]{Corollary}
\newtheorem{lemma}[theorem]{Lemma}

\theoremstyle{definition}

\newtheorem{remark}[theorem]{Remark}

\numberwithin{equation}{section}

\begin{document}


\baselineskip=17pt


\title{On integer linear combinations of terms of rational cycles for the generalized 3x+1 problem}

\author{Yagub N. Aliyev\\
School of IT and Engineering\\ 
ADA University\\
Ahmadbey Aghaoglu str. 61 \\
Baku 1008, Azerbaijan\\
E-mail: yaliyev@ada.edu.az}
\date{}

\maketitle


\renewcommand{\thefootnote}{}



\renewcommand{\thefootnote}{\arabic{footnote}}
\setcounter{footnote}{0}


\begin{abstract}
In the paper, some special linear combinations of the terms of rational cycles of generalized Collatz sequences are studied. It is proved that if the coefficients of the linear combinations satisfy some conditions then these linear combinations are integers. The discussed results are demonstrated on some examples. In some particular cases the obtained results can be used to explain some patterns of digits in $p$-adic representations of the terms of the rational cycles.
\end{abstract}

\section{Introduction} Collatz conjecture or $3x+1$ problem claims that for any positive integer $x_0$, the recursive sequence defined for $n\ge 0$ by $x_{n+1}=S(x_n)=\frac{3x_n+1}{2}$ if $x_n$ is odd and $x_{n+1}=T(x_n)=\frac{x_n}{2}$ if $x_n$ is even, there is a positive integer $N$ such that $x_N=1$ (see \cite{lag}). It is known that this holds true for almost all $x_0$ in the sense of some density. See \cite{terras}, \cite{ev}, \cite{kor}, \cite{tao}, and the references therein for the works in this direction. Another approach is to find all cycles of this recursive sequence. It is conjectured that there are only finitely many such cycles. The only known cycle is the one generated by $x_0=1$. There are more cycles if $x_0$ is allowed to be zero or a negative integer but it is conjectured that their number is also finite (The Finite Cycles Conjecture, see \cite{lag}, \cite{lag2}). The only known non-positive cycles are the ones generated by $x_0=0$, $x_0=-1$, $x_0=-5$, and $x_0=-17$. If a sequence of functions consisted of several $S$ and $T$ is given, then one can speak about rational cycles (see \cite{aliyev}). There is a rational number $x_0$ such that if the operations $S$ and $T$ are applied in the given order, then
the final result is again $x_0$. Rational cycles generated by such $x_0$ have some interesting properties (see \cite{lag2}). Many generalizations of Collatz conjecture were considered by replacing $S$ and $T$ operations by more general $S_k(x)=\frac{p_ix+k}{q}$. One can find such generalizations in \cite{matt}, \cite{matt2}, \cite{levy}. See \cite{con}, \cite{leht}, where similar generalizations are used to prove some results related to undecidability properties. See \cite{laar}, \cite{rouet}, where these generalizations are considered in the context of $2$-adic numbers and $q$ based numeral systems.

In this paper we focus on properties of the terms of rational cycles $x_i$, which show that these rational numbers are "integer like". A linear combination with integer coefficients of two integers is again an integer. In the paper it is proved that under some conditions over the coefficients of the linear combination, the rational numbers $x_i$ and $x_{i+b}$ also form integer linear combinations. Two worked out examples demonstarting the results on particular cases are given. Some applications of these results explaining peculiar patterns of digits in $p$-adic representations of these rational numbers $x_i$ are also discussed.

\section{Notations and lemmas}
Consider composition $P=B_0\circ B_1\circ \ldots \circ B_{n-1}$ of functions $B_i (x)=\frac{p_i x+k_i}{q}$, where $n>1$, $k_i$ are integers, $p_i,q$ are non-zero integers such that $(p_i,q)=1$ for $i=0,1,\ldots ,n-1$. When it is necessary to extend the index $i$, beyond the interval $[0,n-1]$, we suppose that $B_i=B_j$ if $i\equiv j\  (\bmod \ n)$. Consider equation $B_0\circ B_1\circ \ldots \circ B_{n-1}(x)=x$, which can also  be written as
$$
\frac{p_0 \frac{p_1\frac{\frac{p_{n-2} \frac{p_{n-1} x+k_{n-1}}{q}+k_{n-2}}{\begin{matrix}
q\\
\vdots 
\end{matrix}}}{q} +k_1}{q}+k_0}{q} =x.
$$
Note that its solution $x_0$ is a rational number (cf. formula 1.2 and 1.3, \cite{lag2}):
$$
x_0=\frac{p_0p_1\ldots p_{n-2}k_{n-1}+p_0p_1\ldots p_{n-3}k_{n-2}q+\ldots +p_0k_1 q^{n-2}+k_0q^{n-1}}{q^n-p_0p_1\ldots p_{n-1}}.
$$
Similarly, consider equations $B_i\circ B_{i+1}\circ \ldots  \circ B_{i+n-1}(x)=x$ for $i=0,1,\ldots ,n-1$. Note that their solutions $x_i$ are also rational numbers:
$$
x_i=\frac{p_ip_{i+1}\ldots p_{i+n-2}k_{i-1}+p_ip_{i+1}\ldots p_{i+n-3}k_{i-2}q+\ldots +p_ik_{i+1} q^{n-2}+k_iq^{n-1}}{q^n-p_0p_1\ldots p_{n-1}},
$$
where all the indices are taken modulo $n$. In the following, it is assumed that $x_i=x_j$ if $i\equiv j (\bmod n)$. Consider also numbers $U_i=\frac{q^{i}}{q^n-p_0p_1\ldots p_{n-1}},$ for $i=0,1,\ldots ,n$. Note that $U_n=p_0p_1\ldots p_{n-1}U_0+1$. Note also that
$$
x_i=p_ip_{i+1}\ldots p_{i+n-2}k_{i-1}U_0+p_ip_{i+1}\ldots p_{i+n-3}k_{i-2}U_1+\ldots +p_ik_{i+1}U_{n-2}+k_iU_{n-1}.
$$
Note that there are infinitely many pairs of non-zero integers $\alpha,\beta$ and integers $b$, such that $0<b<n$ and $\alpha U_0+\beta U_b$ is an integer or equivalently, $q^n-p_0p_1\ldots p_{n-1}|\alpha +\beta q^b$. Indeed, one can take, for example, $\alpha =k$, $\beta =-kq^{\phi (|q^n-p_0p_1\ldots p_{n-1}|)-1}$, and $b=1$, where $k=1,2,\ldots $ and $\phi$ is Euler's totient function. If $\alpha$ and $\beta$ are fixed then one can ask if such $b$ exists. The answer to this question depends on the choice of $\alpha$ and $\beta$ . For example, if $q=2$, $n=4$, $p_0=3$, $p_1=p_2=p_3=1$, then $q^n-p_0p_1p_2p_3=13$. If $\alpha=1, \beta= 1$ then such $b$ ($0<b<4$) does not exist. But if $\alpha=9, \beta= 1$ then $b=2$ satisfies the condition.
\begin{lemma}{If $\alpha U_0+\beta U_b$ is an integer, then $p_0 p_1\ldots p_{n-1}\beta U_0+\alpha U_{n-b}$ is also an integer.}
\end{lemma}
\begin{proof} Our claim is equivalent to prove that if $q^n-p_0p_1\ldots p_{n-1}|\alpha +\beta q^b$ then $$q^n-p_0p_1\ldots p_{n-1}| p_0 p_1\ldots p_{n-1}\beta+ \alpha q^{n-b}.$$ Indeed, since $(q,q^n-p_0p_1\ldots p_{n-1})=1$ and
$$q^b(p_0 p_1\ldots p_{n-1}\beta+ \alpha q^{n-b})= p_0 p_1\ldots p_{n-1}\beta q^b+ \alpha q^{n},$$
it is sufficient to show that $$q^n-p_0p_1\ldots p_{n-1}| p_0 p_1\ldots p_{n-1}\beta q^b+ \alpha q^{n},$$ which follows from
$$p_0 p_1\ldots p_{n-1}\beta q^b+ \alpha q^{n}=p_0 p_1\ldots p_{n-1}(\alpha+\beta q^b)+ \alpha (q^{n}-p_0 p_1\ldots p_{n-1}).$$
\end{proof}
\begin{lemma}{If $\alpha U_0+\beta U_b$ is an integer, then for $i=0,1,2,\ldots$ the numbers $\alpha U_i+\beta U_{i+b}$ and $p_0 p_1\ldots p_{n-1}\beta U_i+\alpha U_{n+i-b}$ are also integers.}
\end{lemma}
\begin{proof} By mutiplying $\alpha U_0+\beta U_b$ and $p_0 p_1\ldots p_{n-1}\beta U_0+\alpha U_{n-b}$, which are both integers, by integer $q^i$, we obtain that both of the numbers $\alpha U_i+\beta U_{i+b}$ and $p_0 p_1\ldots p_{n-1}\beta U_i+\alpha U_{n+i-b}$ are integers.
\end{proof}
\section{Main results}
\begin{theorem}{If $\alpha U_0+\beta U_b$ is an integer, then for any $i$, satisfying $0\le i<i+b< n$, the number $\alpha x_i+\beta p_i p_{i+1}\ldots p_{i+b-1} x_{i+b}$ is also an integer.}
\end{theorem}
\begin{proof} Note that we can write
$$
\alpha x_i+\beta p_i p_{i+1}\ldots p_{i+b-1} x_{i+b}
$$
$$
=\alpha (p_ip_{i+1}\ldots p_{i+n-2}k_{i-1}U_0+p_ip_{i+1}\ldots p_{i+n-3}k_{i-2}U_1+\ldots +p_ik_{i+1}U_{n-2}+k_iU_{n-1})
$$
$$
+\beta p_i p_{i+1}\ldots p_{i+b-1}(p_{i+b}p_{i+b+1}\ldots p_{i+b+n-2}k_{i+b-1}U_0
$$
$$
+p_{i+b}p_{i+b+1}\ldots p_{i+b+n-3}k_{i+b-2}U_1+\ldots +p_{i+b}k_{i+b+1} U_{n-2}+k_{i+b}U_{n-1})
$$ 
$$
=k_0M_0+k_1M_1+\ldots +k_{n-1}M_{n-1},
$$ 
where
$$
M_j=
    \begin{cases}
        p_i p_{i+1}\ldots p_{j-1} (\alpha U_{n+i-1-j}+\beta p_0 p_1\ldots p_{n-1} U_{i+b-1-j}) & \text{if } i\le j<i+b,\\
        p_i p_{i+1}\ldots p_{n+j-1} (\alpha U_{i-1-j}+\beta U_{i+b-1-j}) & \text{if } 0\le j<i,\\
        p_i p_{i+1}\ldots p_{j-1} (\alpha U_{n+i-1-j}+\beta U_{n+i+b-1-j}) & \text{if } i+b\le j<n.
    \end{cases}
$$
By Lemma 2.1 and Lemma 2.2, all $M_j$, for $j=0,1,\ldots n-1$, are integers and therefore the claim is true.
\end{proof}
\begin{corollary}{If $\alpha U_0+\beta U_b$ is an integer, then for any $i$, the number $\alpha x_i+\beta p_i p_{i+1}\ldots p_{i+b-1} x_{i+b}$ is also an integer.}
\end{corollary}
\begin{proof} 
Since it was assumed that $x_i=x_j$ if $i\equiv j (\bmod n)$, without loss of generality, we can suppose that $0\le i <n$. The case $i+b< n$ was considered in Theorem 2.1. So, we can suppose that $i+b\ge n$. Since $0<b<n$, we have $0\le i+b-n<i$, and therefore $x_{i+b}=x_{i+b-n}$. Consequently,
$$
\alpha x_i+\beta p_i p_{i+1}\ldots p_{i+b-1} x_{i+b}=\alpha x_i+\beta p_i p_{i+1}\ldots p_{i+b-1} x_{i+b-n}.
$$
Multiply this number by $p_{i+b-n}p_{i+b-n+1}\ldots p_{i-1}$, which is relatively prime to  $q^{n}-p_0 p_1\ldots p_{n-1}$, and therefore can not change the property of being or not being an integer for the number $\alpha x_i+\beta p_i p_{i+1}\ldots p_{i+b-1} x_{i+b}$, we obtain
$$
\beta p_0 p_{1}\ldots p_{n-1} x_{i+b-n}+\alpha p_{i+b-n} p_{i+b-n+1}\ldots p_{i-1} x_{i},
$$
which is an integer by Lemma 2.1 and Theorem 3.1.
\end{proof}

\begin{remark} The above results are trivially true for the cases when $b=0$ and $b=n$. Indeed, for the case when $b=0$, if $(\alpha +\beta)U_0$ is an integer, then $(q^n-p_0p_1\ldots p_{n-1})|(\alpha +\beta)$, and therefore $(\alpha +\beta)x_i$ is also an integer for $i=0,1,\ldots, n-1$. For the case when $b=n$, if $\alpha U_0 +\beta U_n=\alpha U_0 +\beta p_0p_1\ldots p_{n-1}U_0+\beta$ is an integer, then $(\alpha  +\beta p_0p_1\ldots p_{n-1})U_0$ is also an integer ($\beta$ is an integer), and therefore $(q^n-p_0p_1\ldots p_{n-1})|(\alpha +p_0p_1\ldots p_{n-1}\beta)$. Consequently $(\alpha +\beta p_0p_1\ldots p_{n-1})x_i$ is an integer for $i=0,1,\ldots, n-1$.
\end{remark}

\begin{corollary}{
If one of the numbers $x_j$ $(j\in \{0, 1,\ldots , n-1\})$ is a fraction with denominator $d$ in its simplest form then all of $x_i$ for $i=0, 1,\ldots , n-1$, are like fractions with the same denominator $d$.}
\end{corollary}
\begin{proof} It was mentioned earlier that one can always take $\alpha =1$, $\beta =-q^{\phi (|q^n-p_0p_1\ldots p_{n-1}|)-1}$, and $b=1$. Then by the main result, the number $x_i+\beta p_i x_{i+1}$ is an integer for $i=0, 1,\ldots , n-1$. Note that $d$ is a divisor of $q^n-p_0p_1\ldots p_{n-1}$, and therefore $d$ is relatively prime to $\beta$ and all $p_i$ for $i=0, 1,\ldots , n-1$. Since $x_j$ is a fraction with denominator $d$, the other numbers $x_{j-1}$, $x_{j-2}$, $x_{j-3},...$ are all like fractions with the same denominator $d$ in their simplest forms.
\end{proof}

\begin{corollary}{
If one of the numbers $x_j$ $(j\in \{0, 1,\ldots , n-1\})$ is an integer then all of $x_i$ for $i=0, 1,\ldots , n-1$, are integers.}
\end{corollary}
\section{Examples}
Let us take $q=3$. Consider composition of functions $P=B_0 \circ B_1 \circ B_2\circ B_3$, where $B_0 (x)=\frac{-5x-2}{3}$, $B_1 (x)=\frac{2x+1}{3}$, $B_2 (x)=\frac{7x+6}{3}$, $B_3 (x)=\frac{-x+3}{3}$. Here $n=4$, $p_0=-5$, $k_0=-2$, $p_1=2$, $k_1=1$, $p_2=7$, $k_2=6$, $p_3=-1$, $k_3=3$, and $q^n-p_0 p_1 p_2 p_3=3^4-(-5)\cdot 2\cdot 7\cdot (-1)=11$.

The solution of equation $B_0 \circ B_1 \circ B_2 \circ B_3 (x)=x$ is the number $x_0=-69/11$. Note that $x_0=x_4$. We can also find the other numbers $x_1=x_5=37/11$, $x_2=50/11$, $x_3=12/11$, by solving the equations $B_1 \circ B_2 \circ B_3 \circ B_0 (x)=x$, $B_2 \circ B_3 \circ B_0 \circ B_1 (x)=x$, $B_3 \circ B_0 \circ B_1 \circ B_2 (x)=x$, respectively. We also find the numbers $U_i=2^i/11$ ($i=0,1,2,3,4$). Note that $4U_0+2U_2=2$ is an integer, which is equivalent to say that $11|(4+2\cdot 3^2)$. So, we can take $\alpha=4$, $\beta=2$, and $b=2$. We observe that $4x_i+2p_i p_{i+1} x_{i+2}$
is an integer for each of $i=0,1,2,4$. Indeed,
\begin{center}
\begin{tabular}{lll}
$4x_0+2p_0 p_1 x_2$ & $=4\cdot (-69/11)+2\cdot (-5)\cdot 2\cdot (50/11)$ &$=-116,$\\
$4x_1+2p_1 p_2 x_3$ & $=4\cdot (37/11)+2\cdot 2\cdot 7\cdot (12/11)$ & $=44,$\\
$4x_2+2p_2 p_3 x_4$ & $=4\cdot (50/11)+2\cdot 7\cdot (-1)\cdot (-69/11)$ & $=106,$\\
$4x_3+2p_3 p_4 x_5$ & $=4\cdot (12/11)+2\cdot (-1)\cdot (-5)\cdot (37/11)$ & $=38,$\\
\end{tabular}
\end{center}
are all integers. 

Now, note that $-5U_0-13U_1=-4$ is also an integer, which is equivalent to say that $11|-44=-5-13\cdot 3^1$. This means that we can take $\alpha =-5$, $\beta=-13$, and $b=1$. So, $-5x_i-13p_i x_{i+1}$
should be an integer for each of $i=0,1,2,3$. Indeed,
\begin{center}
\begin{tabular}{lll}
$-5x_0-13p_0 x_1$ & $=-5\cdot (-69/11)-13\cdot (-5)\cdot (37/11)$ & $=250,$\\
$-5x_1-13p_1 x_2$ & $=-5\cdot (37/11)-13\cdot 2\cdot (50/11)$ & $=-135,$\\
$-5x_2-13p_2 x_3$ & $=-5\cdot (50/11)-13\cdot 7\cdot (12/11)$ & $=-122,$\\
$-5x_3-13p_3 x_4$ & $=-5\cdot (12/11)-13\cdot (-1)\cdot (-69/11)$ & $=-87,$\\
\end{tabular}
\end{center}
are all integers. These observations are in perfect agreement with the main results of the current paper.

\section{Applications} If $p_i\in \{1,p\}$, where $p$ is a nonzero integer and $(p,q)=1$ then there are two type of functions $S_k(x)=\frac{px+k}{q}$ and $T_k(x)=\frac{x+k}{q}$. Let us denote by $m$ the number of $S$ functions in $P$. Then $U_i=\frac{q^{i}}{q^n-p^m},$ for $i=0,1,\ldots ,n$. Denote by $\sigma(i,j)$ the number of $S$ functions in the fragment $B_iB_{i+1}\ldots B_{j-1}$ of $P$. In particular, $\sigma(i,i)=0$, because it corresponds to the empty fragment of $P$. Let $x_i$ be the solution of the equation $B_i\circ B_{i+1} \circ B_{i+2} \circ \ldots \circ B_{i+n-1}(x)=x$, where all the indices are taken modulo $n$. Take $\alpha=p^l$ for some non-negative integer $l$, and $\beta= -1$. For this special case the main result of the current paper can be written as $p^l x_i- p^{\sigma(i,i+b)}x_{i+b}\in \mathbb{Z}$. This can be visualized by writing $x_i$ as $p$-adic numbers in a table and noting that $p$-adic digits at the corresponding place values of $p^lx_i$ and $p^{\sigma(i,i+b)}x_{i+b}$ are identical, except for finitely many digits at lower place values. Let us demonstrate this on an example. Let $q=2$, $p=11$, $P=B_0 \circ B_1 \circ B_2 \circ B_3 \circ B_4 \circ B_5 \circ B_6$, where $B_0 (x)=B_1 (x)=B_2 (x)=B_3 (x)=B_5 (x)=T_0(x)$, $B_4 (x)=S_5(x)$, and, $B_6 (x)=S_3(x)$. Here $n=7$, $m=2$, $q^n-p^m=2^7-11^2=7$ and $U_i=2^i/7$ ($i=0,1,2,\ldots$). Note that $$U_0-U_3=-1,\ \ 
11U_0-U_2=1,\ \ 11^2U_0-U_1=17,\ \ 11^3U_0-U_0=190,$$ are all integers, which is equivalent to say that $$7|(1-2^3),\ \ 7|(11-2^2),\ \ 7|(11^2-2^1),\ \ 7|(11^3-2^0),$$ respectively. By the main result of the current paper we can say that $$x_i-11^{\sigma(i,i+3)}x_{i+3},\ \ 11x_i-11^{\sigma(i,i+2)}x_{i+2},\ \ 11^2x_i-11^{\sigma(i,i+1)}x_{i+1},\ \ 11^3x_i-x_{i},$$ are also integers for $i=0,1,2,\ldots$. The functions $B_i$, the numbers $x_i$, their 11-adic representations (the letter A means digit 10), and the patterns formed by the digits can be seen in the following table:
\begin{center}
\begin{tabular}{ll|ll|ll|ll|ll|ll|lll}
$x_0=53/7$ & $=\ldots$   &7 & 9 & 4 & 7 & 9 & 4 & 8 & 6 &   &    & &   & $S_3=B_6$ \\
$x_6=302/7$ & $=\ldots$ &9 & 4 & 7 & 9 & 4 & 7 & 9 & 8 & 7 &    & &   & $T_0=B_5$ \\
$x_5=151/7$ & $=\ldots$ &4 & 7 & 9 & 4 & 7 & 9 & 4 & 9 & 9 &    & &   & $S_5=B_4$ \\
$x_4=848/7$ & $=\ldots$ &7 & 9 & 4 & 7 & 9 & 4 & 7 & A & 4 & 8 & &   &  $T_0=B_3$\\
$x_3=424/7$ & $=\ldots$ &9 & 4 & 7 & 9 & 4 & 7 & 9 & 5 & 2 & 4 & &   & $T_0=B_2$ \\
$x_2=212/7$ & $=\ldots$ &4 & 7 & 9 & 4 & 7 & 9 & 4 & 8 & 1 & 2 & &   & $T_0=B_1$ \\
$x_1=106/7$ & $=\ldots$ &7 & 9 & 4 & 7 & 9 & 4 & 7 & 9 & 6 & 1 & &   &  $T_0=B_0$\\
$x_0=53/7$ & $=\ldots$   &9 & 4 & 7 & 9 & 4 & 7 & 9 & 4 & 8 & 6 & &   &  \\
\end{tabular}
\end{center}
More examples of such patterns and applications to the original $3x+1$ problem are given in \cite{aliyev} and the references therein.

The Finite Cycles Conjecture mentioned at the beginning of this paper claims that the only integer cycles for $3x+1$ problem are the ones generated by $x_0=0$, $x_0=-1$, $x_0=1$, $x_0=-5$, and $x_0=-17$. These numbers correspond to compositions $P_1=T$, $P_2=S$, $P_3=T\circ S$, $P_4=T\circ S\circ S$, and $P_5=T\circ T\circ T\circ S\circ S\circ S\circ T\circ S\circ S\circ S\circ S$. For these compositions $q=2$, $p_i=1$ or $3$, and the numbers in the following table are calculated.
\begin{center}
\begin{tabular}{|l|l|l|l|}
\hline
$x_0$ & $P$   &$n$ & $q^n-p_0p_1\ldots p_{n-1}$  \\
\hline
$0$ & $P_1$ &1 & $2^1-1=1$   \\
$-1$ & $P_2$ &1 & $2^1-3=-1$   \\
$1$ & $P_3$ &2 &$2^2-3=1$ \\
$-5$ & $P_4$ &3 & $2^3-3\cdot 3=-1$ \\
$-17$ & $P_5$ &11 & $2^{11}-3^7=-139$  \\
\hline
\end{tabular}
\end{center}

These compositions with integer $x_0$ also show that the main results of the current paper, namely Theorem 3.1 and Corollary 3.2 can not be written as "iff" statements. Indeed, if the numbers $x_i$ $(i\in \{0, 1,\ldots , n-1\})$ are integers then $\alpha x_i+\beta p_i p_{i+1}\ldots p_{i+b-1} x_{i+b}$ is an integer for any choice of integers $\alpha ,\beta , b$, which is not the case for $\alpha U_0+\beta U_b$. Nevertheless, Lemma 2.1 can be written as an "iff" statement.

The composition $P_5$ is different from the others in the sense that $q^n-p_0p_1\ldots p_{n-1}\ne \pm 1$ but $x_0$ is still an integer. The compositions $P_i^k$ ($i=1,2,\ldots ,5$; $k=1,2,\ldots$), defined recursively by $P_i^1=P_i$ and $P_i^{k+1}=P_i^k\circ P_i$ ($i=1,2,\ldots ,5$; $k=1,2,\ldots$) also have integer $x_0$. Determination of all such compositions with integer $x_0$ or proving that all other compositions correspond to non-integer rational $x_0$ will be helpful for the solution of $3x+1$ problem. The results of the current paper shed some light on the structure of rational cycles and therefore they might be useful for this purpose.
\section{Conclusion}

In the paper some generalizations of Collatz conjecture or $3x+1$ problem are studied. Some results are obtained proving that special linear combinations of the terms of rational cycles are integers. Demonstrations of these results on some concrete examples are given. These results are then used to explain some patterns of digits in $p$-adic representations of the rational cycles.
\subsection*{Acknowledgements}


\begin{thebibliography}{HD82}




\normalsize
\baselineskip=17pt
 
\bibitem{aliyev} Y.N. Aliyev, \textit{The $3x+1$ Problem For Rational Numbers : Invariance of Periodic Sequences in $3x+1$ Problem}, 2020 IEEE 14th International Conference on Application of Information and Communication Technologies (AICT), Tashkent, Uzbekistan, 2020, 1-4. \url{https://doi.org/10.1109/AICT50176.2020.9368585}

\bibitem{con} J.H. Conway, \textit{On Unsettleable Arithmetical Problems}, The American Mathematical Monthly 120(3), 192-198 (2013). \url{https://doi.org/10.4169/amer.math.monthly.120.03.192}

\bibitem{ev} C.J. Everett,
Iteration of the number-theoretic function $f(2n) = n, f(2n + 1) = 3n + 2$,
Advances in Mathematics, 25(1) (1977) 42-45. \url{https://doi.org/10.1016/0001-8708(77)90087-1}

\bibitem{kor} I. Korec, \textit{A density estimate for the $3x+1$ problem}, Math. Slovaca 44 (1994), no. 1, 85–89. \url{https://eudml.org/doc/32414}

\bibitem{laar} T. Laarhoven, B. de Weger, \textit{The Collatz conjecture and De Bruijn graphs}, Indagationes Mathematicae, 24(4), 2013, 971-983. \url{https://doi.org/10.1016/j.indag.2013.03.003}

\bibitem{lag} J.C. Lagarias (Ed.), \textit{The Ultimate Challenge: The $3x+1$ Problem}, 344 pp., AMS, 2010.

\bibitem{lag2} J.C. Lagarias, \textit{The set of rational cycles for the $3x + 1$ problem}, Acta Arith. 56 (1990) 33–53. \url{http://eudml.org/doc/206298}

\bibitem{leht} E. Lehtonen, \textit{Two undecidable variants of Collatz's problems}, Theor. Comput. Sci., 407, 1–3 (2008) 596–600. \url{https://doi.org/10.1016/j.tcs.2008.08.029}

\bibitem{levy} D. Levy, \textit{Injectivity and surjectivity of Collatz functions}, Discrete Mathematics, 285, Issues 1–3,
2004, 191-199. \url{https://doi.org/10.1016/j.disc.2004.03.008}

\bibitem{matt} K. Matthews, A. Watts, \textit{A generalization of Hasse's generalization of the Syracuse algorithm}, Acta Arithmetica 43.2 (1984) 167-175. \url{http://eudml.org/doc/205897}

\bibitem{matt2} K. Matthews, A. Watts, \textit{A Markov approach to the generalized Syracuse algorithm}, Acta Arithmetica 45.1 (1985): 29-42. \url{http://eudml.org/doc/205953}

\bibitem{rouet} J.L. Rouet, M.R. Feix, \textit{A Generalization of the Collatz Problem. Building Cycles and a Stochastic Approach}, Journal of Statistical Physics 107, 1283–1298 (2002). \url{https://doi.org/10.1023/A:1015122111180}

\bibitem{tao} T. Tao,  \textit{Almost all orbits of the Collatz map attain almost bounded value}, Forum of Mathematics, Pi, 10, E12, 2022. \url{https://doi.org/10.1017/fmp.2022.8}

\bibitem{terras} R. Terras, \textit{A stopping time problem on the positive integers}, Acta Arith.30.3 (1976), 241. \url{http://eudml.org/doc/205476}



\end{thebibliography}
\end{document}